\documentclass{article}
\usepackage{amsmath}
\usepackage{amssymb}
\usepackage{esint}
\usepackage{geometry}
\usepackage{cleveref}

\makeatletter
 \usepackage{amsfonts}
\usepackage{color}
\usepackage{graphicx}
\providecommand{\U}[1]{\protect \rule{.1in}{.1in}}

\newtheorem{theorem}{Theorem}[section]

\newtheorem{corollary}[theorem]{Corollary}

\newtheorem{definition}[theorem]{Definition}
\newtheorem{example}[theorem]{Example}

\newtheorem{lemma}[theorem]{Lemma}

\newtheorem{proposition}[theorem]{Proposition}
\newtheorem{remark}[theorem]{Remark}

\newenvironment{proof}[1][Proof]{\noindent \textbf{#1.} }{\  \rule{0.5em}{0.5em}}
\usepackage{graphicx}
\usepackage{subfigure}

\def\lt{\left}
\def\rt{\right}

\def\br{\mathbb{R}}

\def\cp{\mathcal{P}}
\def\cf{\mathcal{F}}

\def\cf{\mathcal{F}}

\def\pe{{\mathbb{E}^{\cp}}}

\def\bbl{{\boldsymbol{\lambda}}}
\def\bbm{\boldsymbol{\mu}}

\def\bbk{\boldsymbol{\kappa}}

\def\ccp{{\rm{co}(\cp)}}

\def\ccp{{\rm{co}(\cp)}}

\makeatother

\begin{document}

\title{On the calculation of upper variance under multiple probabilities
\footnote{This work was supported by NSF of Shandong Provence (No.ZR2021MA018),  NSF of China (No.11601281),  National Key R\&D Program of China (No.2018YFA0703900) and the Qilu Young Scholars Program of Shandong University.}}

\author{Xinpeng Li\footnote{Research Center for Mathematics and Interdisciplinary Sciences, Shandong University, 266237, Qingdao, China.
  Email:lixinpeng@sdu.edu.cn} \ \ \
 Miao Yu\footnote{Research Center for Mathematics and Interdisciplinary Sciences, Shandong University, 266237, Qingdao, China
 .}  \ \ \
 Shiyi Zheng\footnote{School of Mathematics, Shandong University, 250100, Jinan, China.}
}
\date{ }

\maketitle

\abstract{The notion of upper variance under multiple probabilities is defined by a corresponding minimax optimization problem. This paper proposes a simple algorithm  to solve the related minimax optimization problem exactly. As an application, we provide the probabilistic representation for a class of quadratic programming problems. }



\section{Introduction}
The notion of upper and lower variances is a generalization of classical notion of variance to the setting of multiple probabilities or the problems with uncertainty. For example, let $X$ be a random variable representing the daily return of one stock. Without loss of generality, we assume $X\sim N(0.1,0.4)$ in the bull market and $X\sim N(-0.1,0.4)$ in the bear market respectively. In fact, the parameters of mean and variance can be estimated from the historical data of the stock. We hope to estimate the ``risk" (variance) of the stock after one month, but we do not know the true state (bull or bear) of the stock market in the future. In this case, the notion of upper and lower variances is a good measure for such risk. More precisely, let $\cp=\{P_1,\cdots,P_K\}$ be a set of probability measures on measurable space $(\Omega,\cf)$. For each random variable $X$ with $\max_{1\leq i\leq K}E_{P_i}[X^2]<\infty$, the upper and lower variances of $X$ under $\cp$ are defined respectively by
$$\overline{V}(X)=\min_{\mu\in\br}\max_{1\leq i\leq K}E_{P_i}[(X-\mu)^2], \ \ \ \ \underline{V}(X)=\min_{\mu\in\br}\min_{1\leq i\leq K}E_{P_i}[(X-\mu)^2].$$
Walley \cite{walley1991statistical} pointed out that it is quite easy to calculate the lower variance, but
more difficult to calculate the upper variance (see Appendix G in \cite{walley1991statistical}). This paper aims to provide a simple algorithm to calculate the upper variance $\overline{V}(X)$.

The upper and lower variances can be widely used in many fields, especially in mathematical finance. Li et al. \cite{li2022upper} used upper and lower variances to obtain a general $G$-Var model with mean-uncertainty, which generalized the $G$-Var model with zero-mean in Peng et al. \cite{pyy}. In this paper, we provide a new application of upper variance. Since the upper variance $\overline{V}(X)$ can be regarded as the supremum of the variance over the convex hull of $\cp$ (see Theorem \ref{vpt}), it can be calculated by a quadratic programming problem on the unit simplex in $\br_+^K$ (see Proposition \ref{p1}). Our algorithm is also useful for solving a class of quadratic programming problems, and we can obtain a probabilistic representation for such quadratic programming problems, i.e., the related optimal value is the upper variance under multiple probabilities.

  \par The structure of this paper is as follows. Section 2 recalls the notions and properties of upper and lower variances. The algorithm for the calculation of upper variance is presented in Section 3. Section 4 shows the application of upper variance in quadratic programming.

\section{Preliminaries}

We firstly recall the preliminaries of sublinear expectation theory introduced in Peng \cite{peng2010nonlinear}, which is a powerful tool to deal with problems with uncertainty, i.e., problems under multiple probabilities.

Let $\Omega$ be a given set and $(\Omega,\cf)$ be a measurable space. Let $\cp$ be a set of probability measures on $(\Omega,\cf)$ characterizing Knightian uncertainty. We define the corresponding upper expectation $\pe$ by
$$\pe[X]=\sup_{P\in\cp}E_P[X].$$
Obviously, $\pe$ is a sublinear expectation satisfying
  \begin{itemize}
    \item [(i)]Monotonicity: $\pe[X]\le \pe[Y],$ if $X\le Y$;
    \item [(ii)]Constant preserving : $\pe[c]=c, \forall c\in \mathbb{R}$;
    \item [(iii)]Sub-additivity : $\pe[X+Y]\le \pe[X]+\pe[Y]$;
    \item [(iv)]Positive homogeneity : $\pe[\lambda X]=\lambda \pe[X], \forall \lambda \geq 0$.
  \end{itemize}
We call $(\Omega,\cf,\pe)$ the sublinear expectation space.

In this paper, we denote $\ccp$ the convex hull of $\cp$. It is clear that
  $$\pe[X]=\sup_{P\in\cp}E_P[X]=\sup_{P\in\ccp}E_P[X].$$
The upper expectation $\pe[X]$ and the lower expectation $-\pe[-X]$ of $X$ are denoted by $\overline{\mu}_X$ and $\underline{\mu}_X$ respectively, called upper mean and lower mean of $X$. The interval
 $[\underline{\mu}_X,\overline{\mu}_X]$ characterizes the mean-uncertainty of $X$, denoted by $M_X$.

Now we recall the notion of upper and lower variances which was first introduced by Walley \cite{walley1991statistical} for bounded random variable under coherent prevision, and then generalized by Li et al. \cite{li2022upper}.

Let $(\Omega,\cf,P)$ be a probability space and $X$ be a random variable with finite second moment. The variance of $X$, denoted by $V_P(X)$, is defined as
 $$V_P(X)=E_P[(X-E_P[X])^2].$$
 We note that
 \begin{equation*}
 E_P[(X-\mu)^2]=V_P(X)+(E_P[X]-\mu)^2,
 \end{equation*}
immediately,
 $$V_P(X)=\min_{\mu\in\br}E_P[(X-\mu)^2],$$
then we use $\pe$ instead of $E_P$ to obtain the following definition.

\begin{definition}
    For a random variable $X$ on sublinear expectation space $(\Omega,\mathcal{F},\pe)$ with $\pe[X^2]
    <\infty$, we define the upper variance of $X$ as
    \[\overline{V}(X): = \min_{\mu\in M_X} \{\pe[(X-\mu)^2]\},
     \]
     and the lower variance of $X$ as
     \[ \underline{V}(X): = \min_{\mu\in M_X} \{-\pe[-(X-\mu)^2]\}.
       \]
 \end{definition}
\begin{remark}
Since $\pe[(X-\mu)^2]$ is a strict convex function of $\mu$, there exists a unique $\mu^*\in M_X$ such that
$$\overline{V}(X)=\pe[(X-\mu^*)^2].$$
\end{remark}
It is not hard to prove that
$$\overline{V}(X)=\min_{\mu\in\mathbb{R}}\{\pe[(X-\mu)^2]\}, \ \ \ \underline{V}(X)=\min_{\mu\in\mathbb{R}}\{-\pe[-(X-\mu)^2]\}.$$
More details can be found in Walley \cite{walley1991statistical} and Li et al. \cite{li2022upper}.

 Thanks to the minimax theorem, we have the following variance envelop theorem.
 \begin{theorem}\label{vpt}
    Let $X$ be a random variable on  $(\Omega,\mathcal{F},\pe)$ with
    $\pe[X^2]<\infty$. Then we have
    \begin{itemize}
      \item[(i)] $\overline{V}(X)=\sup_{P\in\ccp}V_P(X)$.
      \item[(ii)] $\underline{V}(X)=\inf_{P\in\ccp}V_P(X)=\inf_{P\in\cp}V_P(X)$.
    \end{itemize}
  \end{theorem}
\begin{proof}
Since the interval $M_X$ is convex and compact, and $\ccp$ is convex, furthermore, $E_P[(X-\mu)^2]$ is a linear function in $P$ and a convex function in $\mu$, by minimax theorem in Sion \cite{sion1958general}, we have
\begin{align*}
\overline{V}(X)&=\min_{\mu\in M_X}\sup_{P\in\cp}E_P[(X-\mu)^2]\\
&=\min_{\mu\in M_X}\sup_{P\in\ccp}E_P[(X-\mu)^2]\\
&=\sup_{P\in\ccp}\min_{\mu\in M_X}E_P[(X-\mu)^2]=\sup_{P\in\ccp}V_P(X).
\end{align*}

It is obvious that
\begin{align*}
\underline{V}(X)&=\min_{\mu\in M_X}\inf_{P\in\cp}E_P[(X-\mu)^2]\\
&=\min_{\mu\in M_X}\inf_{P\in\ccp}E_P[(X-\mu)^2]\\
&=\inf_{P\in\cp}\min_{\mu\in M_X}E_P[(X-\mu)^2]=\inf_{P\in\cp}V_P(X)=\inf_{P\in\ccp}V_P(X).
\end{align*}
\end{proof}

 \begin{remark}

Unlike the envelop theorems in Walley \cite{walley1991statistical} and Li et al. \cite{li2022upper}, we do not require that $\cp$ to be weakly compact. But the convexity of the $\cp$ is necessary, see Example \ref{em1}.
 \end{remark}

\begin{example}\label{em1}
Let $X$ be normally distributed with $X\sim N(0.1,0.4)$ under $P_1$ and $X\sim N(-0.1,0.4)$ under $P_2$. Taking $\cp=\{P_1,P_2\}$, we obtain
$$\overline{V}(X)=V_{P^*}(X)=0.41>0.4=\max\{V_{P_1}(X),V_{P_2}(X)\},$$
where $P^*=\frac{1}{2}(P_1+P_2)$, and
$$\underline{V}(X)=0.4=V_{P_1}(X)=V_{P_2}(X).$$
\end{example}

In this paper, we only consider the case of $\cp$ consisting of finitely many probability measures, i.e., $\cp=\{P_1,\cdots, P_K\}$. In this case,
$$\ccp=\{P_{\bbl}: P_{\bbl}=\lambda_1P_1+\cdots+\lambda_KP_K, \forall {\bbl}=(\lambda_1,\cdots,\lambda_K)^T\in \Delta^K\},$$
where $\Delta^K=\{{\bbl}\in\br^K: \lambda_1+\cdots+\lambda_K=1,
 \lambda_i\geq 0, 1\leq i\leq K\}$.

As pointed out by Walley \cite{walley1991statistical}, it is usually quite easy to calculate lower variance $\underline{V}(X)$ by Theorem \ref{vpt}, because $\underline{V}(X)=\min_{1\leq i\leq K}V_{P_i}(X)$.
For the upper variance $\overline{V}(X)$, there exists $\bbl^*\in\Delta^K$ such that $\overline{V}(X)=V_{P_{\bbl^*}}(X)$, but $P_{\bbl^*}$ is not an extreme point in general (see Example \ref{em1}), so the calculation of upper variance $\overline{V}(X)$ is more difficult.

\section{Calculation of upper variance}

In this section, we propose a simple algorithm to calculate the upper variance $\overline{V}(X)$ under $\cp=\{P_1,\cdots, P_K\}$.

Our goal is to solve the following minimax problem:
\begin{equation}\label{eee1}
\overline{V}(X)=\min_{\mu\in M_X}\max_{1\leq i\leq K}E_{P_i}[(X-\mu)^2].
\end{equation}

We rewrite (\ref{eee1}) as
\begin{equation}\label{eee3}
\overline{V}(X)=\min_{\mu\in M_X}\max_{1\leq i\leq K}(\mu^2-2\mu_i\mu+\kappa_i),
\end{equation}
where $\mu_i=E_{P_i}[X]$ and $\kappa_i=E_{P_i}[X^2]$, $1\leq i\leq K$.

Without loss of generality, we assume $\mu_1\leq\mu_2\leq\cdots\leq\mu_K$.

In order to prove our main theorem, we need following two lemmas. The first lemma characterizes the position of the optimal point for (\ref{eee1}). The second lemma calculates the upper variance for two probability measures.
\begin{lemma}\label{p31}
The unique optimal $\mu^*\in M_X$ in (\ref{eee1}) should satisfy one of the following two conditions.

(1) $\mu^*$ is the minimum of some parabola $f_i(\mu):=\mu^2-2\mu_i\mu+\kappa_i$.

(2) $\mu^*$ is the intersection  of two parabolas $f_i$ and $f_j$, i.e., $f_i(\mu^*)=f_j(\mu^*)$.
\end{lemma}

\begin{proof}
If $\mu^*$ is not the intersection of two parabolas, without loss of generality, we assume that
$$f_1(\mu^*)<f_2(\mu^*)<\cdots<f_K(\mu^*).$$
There exists a neighbourhood $O$ of $\mu^*$ such that
$$f_1(x)<f_2(x)<\cdots<f_K(x), \ \ \forall x\in O.$$
Thus
$$\min_{x\in O}\max_{1\leq i\leq K}f_i(x)=\min_{x\in O}f_K(x),$$
$\mu^*$ is the minimum of $f_K$.
\end{proof}

\begin{lemma}\label{lemma32}
The upper variance of $X$ under $\cp=\{P_1, P_2\}$ can be calculated by
$$\overline{V}(X)=\max\{\kappa_1-\mu_1^2,\kappa_2-\mu_2^2,h(\mu_{12})\},$$
where  \begin{align*}
 \mu_{12}=
 \begin{cases}
    (\mu_1\vee\frac{\kappa_{2}-\kappa_{1}}{2(\mu_{2}-\mu_{1})})\wedge\mu_2, & \mu_1<\mu_2, \\
    \mu_1, & \mu_1=\mu_2,
 \end{cases}
\end{align*}
and $h(x)=x^2-2\mu_1x+\kappa_1$.

\end{lemma}

\begin{proof}
We consider the optimal $\mu^*$ in two cases as in Lemma \ref{p31}.

In Case (1), there exists $i_0\in\{1,2\}$ such that $\overline{V}(X)=V_{P_{i_0}}(X)=\kappa_{i_0}-\mu_{i_0}^2$ and $\mu^*=\mu_{i_0}$.

In Case (2), since $f_{1}(\mu^*)=f_{2}(\mu^*)$, we have
$$\mu^*=\frac{\kappa_{2}-\kappa_{1}}{2(\mu_{2}-\mu_{1})},$$
and
\begin{align*}
\overline{V}(X)=h(\mu^*)=\frac{\mu_{2}\kappa_{1}-\mu_{1}\kappa_{2}}{\mu_{2}-\mu_{1}}+\frac{(\kappa_{1}-\kappa_{2})^2}{4(\mu_{1}-\mu_{2})^2}.
\end{align*}
In this case, we further require that
$$\mu^*\in[\mu_{1},\mu_{2}], \ \ \ \ \mu_{1}<\mu_{2}.$$

%
%
%

In fact, if $\frac{\kappa_{2}-\kappa_{1}}{2(\mu_{2}-\mu_{1})}\notin[\mu_1,\mu_2]$, then we have
$$h(\mu_{1})\vee h(\mu_2)\leq\max\{\kappa_1-\mu_1^2,\kappa_2-\mu^2_2\}.$$
We  take $\mu_1$ or $\mu_2$ instead of $\frac{\kappa_{2}-\kappa_{1}}{2(\mu_{2}-\mu_{1})}$ in this case.

Finally, we obtain
$$\overline{V}(X)=\max\{\kappa_1-\mu_1^2,\kappa_2-\mu_2^2,h(\mu_{12})\}.$$

\end{proof}

\begin{theorem}\label{th32}
The upper variance of $X$ under $\cp=\{P_1,\cdots, P_K\}$ can be calculated by
$$\overline{V}(X)=\max\lt\{\max_{1\leq i\leq K}(\kappa_i-\mu_i^2),\max_{1\leq i<j\leq K}h_{ij}(\mu_{ij})\rt\},$$
where  \begin{align*}
 \mu_{ij}=
 \begin{cases}
    (\mu_i\vee\frac{\kappa_{j}-\kappa_{i}}{2(\mu_{j}-\mu_{i})})\wedge\mu_j, & \mu_i<\mu_j, \\
    \mu_i, & \mu_i=\mu_j,
 \end{cases}  \ \ \ \ \ 1\leq i<j\leq K
\end{align*}
and $h_{ij}(x)=x^2-2\mu_ix+\kappa_i$.
\end{theorem}
\begin{proof}

For $1\leq i<j\leq K$, let $\overline{V}_{ij}(X)$ be the upper variance under two probability measures $P_i$ and $P_j$. Then by Lemma \ref{lemma32}, we have
$$\overline{V}_{ij}(X)=\max\{\kappa_i-\mu_i^2, \kappa_j-\mu_j^2, h_{ij}(\mu_{ij})\}.$$

It is obvious that $\overline{V}(X)\geq \overline{V}_{ij}(X), \ 1\leq i<j\leq K$, we obtain
$$\overline{V}(X)\geq\max\lt\{\max_{1\leq i\leq K}(\kappa_i-\mu_i^2),\max_{1\leq i<j\leq K}h_{ij}(\mu_{ij})\rt\}.$$

We note that $\kappa_i-\mu_i^2=\min_{\mu\in\br}f_i(\mu)$, $1\leq i\leq K$ and $h_{ij}(\mu_{ij})=f_i(\mu_{ij})=f_j(\mu_{ij})$ if $\mu_{ij}=\frac{\kappa_j-\kappa_i}{2(\mu_j-\mu_i)}$, $1\leq i<j\leq K$.

By Lemma \ref{p31},
$$\overline{V}(X)\leq\max\lt\{\max_{1\leq i\leq K}(\kappa_i-\mu_i^2),\max_{1\leq i<j\leq K}h_{ij}(\mu_{ij})\rt\}.$$

\end{proof}

\begin{corollary}
The upper variance of $X$ under $\cp=\{P_1,\cdots, P_K\}$ can be calculated by
$$\overline{V}(X)=\max_{1\leq i<j\leq K}\lt\{\overline{V}_{ij}(X)\rt\},$$
where $\overline{V}_{ij}(X)$ is the upper variance under $P_i$ and $P_j$, $1\leq i<j\leq K$.
\end{corollary}

Now we present our algorithm as following. Given a random variable $X$ under $\cp=\{P_1,\cdots, P_K\}$, we calculate $\mu_i=E_{P_i}[X]$ and $\kappa_i=E_{P_i}[X^2]$, $1\leq i\leq K$.

\noindent{\bf{Algorithm}:}

\noindent Step (1): We sort $\mu_1\leq \mu_2\leq \cdots\leq \mu_K$.

\noindent Step (2): For $1\leq i<j\leq K$, we calculate $\mu_{ij}$ as
 \begin{align*}
 \mu_{ij}=
 \begin{cases}
    (\mu_i\vee\frac{\kappa_{j}-\kappa_{i}}{2(\mu_{j}-\mu_{i})})\wedge\mu_j, & \mu_i<\mu_j, \\
    \mu_i, & \mu_i=\mu_j.
 \end{cases}
\end{align*}

\noindent Step (3): Output
$$\overline{V}(X)=\max\lt\{\max_{1\leq i\leq K}(\kappa_i-\mu_i^2),\max_{1\leq i<j\leq K}h_{ij}(\mu_{ij})\rt\},$$
where $h_{ij}(x)=x^2-2\mu_ix+\kappa_i$.

In addition, the lower variance is calculated simply by
$$\underline{V}(X)=\min_{1\leq i\leq K}(\kappa_i-\mu_i^2).$$

In practice, $\mu_i$ and $\kappa_i$ can be easily estimated from data. For example, let $X$ be the daily return of one stock. In the real market, we can obtain the daily return data $\{x_i\}_{i\in I}$ (resp. $\{x_j\}_{j\in J}$) from bull (resp. bear) market, where $I$ (resp. $J$) denote the periods of bull (resp. bear) market. The we can estimate the sample means and variances as
$$\hat{\mu}_1=\frac{\sum_{i\in I}x_i}{|I|}, \ \ \ \hat{\mu}_2=\frac{\sum_{j\in J}x_j}{|J|},$$
and
$$\hat{\sigma}^2_1=\frac{\sum_{i\in I} (x_i-\hat{\mu}_1)^2}{|I|-1}, \ \ \ \hat{\sigma}^2_2=\frac{\sum_{j\in J} (x_j-\hat{\mu}_2)^2}{|J|-1}.$$

Then we take $\mu_i=\hat{\mu}_i$ and $\kappa_i=\hat{\sigma}^2_i+\mu_i^2$ ($i=1,2$) to calculate the upper and lower variances.


\section{Application: Quadratic programming}

By Theorem \ref{vpt}, (\ref{eee1}) is equivalent to the following convex quadratic programming problem.

\begin{proposition}\label{p1}
\begin{equation}\label{eee2}
\overline{V}(X)=\max_{\bbl\in\Delta^K}(\bbl^T\bbk-(\bbl^T\bbm)^2),
\end{equation}
where $\bbk=(E_{P_1}[X^2], \cdots, E_{P_K}[X^2])^T$ and  $\bbm=(E_{P_1}[X], \cdots, E_{P_K}[X])^T$.
\end{proposition}
\begin{proof}
By Theorem \ref{vpt}, we have
$$\overline{V}(X)=\max_{\bbl\in\Delta^K}V_{P_\bbl}(X).$$
It is easily seen that
\begin{align*}
V_{P_\bbl}(X)=E_{P_\bbl}[X^2]-(E_{P_\bbl}[X])^2=\bbl^T\bbk-(\bbl^T\bbm)^2.
\end{align*}
\end{proof}

It can be seen that (\ref{eee2}) is a convex quadratic programming problem. There are many numerical algorithms to solve such a problem, e.g., the polynomial-time interior-point algorithm (Nesterov and Nemirovski \cite{Nesterov1994InteriorpointPA}) and accelerated gradient method (Nesterov \cite{Nesterov2004}), etc. However,  we only obtain approximate solutions by most existing algorithms. In this section, we provide a simple method to solve such a problem exactly by the means of upper variance under multiple probabilities.

We have the following theorem.

\begin{theorem}
Given $\bbm=(\mu_1,\cdots,\mu_K)^T\in\br^K$ with $\mu_1\leq\mu_2\leq\cdots\leq\mu_K$ and $\bbk=(\kappa_1,\cdots,\kappa_K)^T\in\br^K$, the solution of the following maximize problem:
\begin{equation*}
V=\max_{\bbl\in\Delta^K}(\bbl^T\bbk-(\bbl^T\bbm)^2)
\end{equation*}
is given by
$$V=\max\lt\{\max_{1\leq i\leq K}(\kappa_i-\mu_i^2),\max_{1\leq i<j\leq K}h_{ij}(\mu_{ij})\rt\},$$
where  \begin{align*}
 \mu_{ij}=
 \begin{cases}
    (\mu_i\vee\frac{\kappa_{j}-\kappa_{i}}{2(\mu_{j}-\mu_{i})})\wedge\mu_j, & \mu_i<\mu_j, \\
    \mu_i, & \mu_i=\mu_j,
 \end{cases}  \ \ \ \ \ 1\leq i<j\leq K
\end{align*}
and $h_{ij}(x)=x^2-2\mu_ix+\kappa_i$.

If there exists $i_0$ such that $V=\kappa_{i_0}-\mu_{i_0}^2$, then the optimal $\bbl^*$ is given by $\lambda_{i_0}^*=1$ and $\lambda_{j}^*=0$, $j\neq i_0$. Otherwise, there exists $1\leq i_0<j_0\leq K$ such that $V=h_{i_0j_0}(\mu_{i_0j_0})$, then the optimal $\bbl^*$ is given by $\lambda_{i_0}^*=\frac{\mu_{j_0}}{\mu_{j_0}-\mu_{i_0}}+\frac{\kappa_{i_0}-\kappa_{j_0}}{2(\mu_{i_0}-\mu_{j_0})^2}$, $\lambda_{j_0}^*=1-\lambda_{i_0}^*$ and $\lambda_j^*=0$, $j\neq i_0,j_0$.

\end{theorem}
\begin{proof}
We only give the sketch of the proof. Let $C=\min_{1\leq i\leq K}\{\kappa_i-\mu_i^2\}$.

(1) $C>0$.  In this case, (\ref{eee3}) is equivalent to (\ref{eee1}). $V$ can be calculated by Theorem \ref{th32}. We only need to consider the optimal $\bbl^*$.

If there exists $i_0$ such that $V=\kappa_{i_0}-\mu_{i_0}^2$, then it is clear that $\lambda_{i_0}^*=1$ and $\lambda_{j}^*=0$ for $j\neq i_0$.

If there exists $1\leq i_0<j_0\leq K$ such that $V=h_{i_0j_0}(\mu_{i_0j_0})$, we consider the following maximize problem:
\begin{equation}\label{eee4}
\max_{\lambda_{i_0}+\lambda_{j_0}=1} \left\{\lambda_{i_0}\kappa_{i_0}+\lambda_{j_0}\kappa_{j_0}-(\lambda_{i_0}\mu_{i_0}+\lambda_{j_0}\mu_{j_0})^2\right\}.
\end{equation}
The optimal solution of (\ref{eee4}) is
$$\lambda_{i_0}=\frac{\mu_{j_0}}{\mu_{j_0}-\mu_{i_0}}+\frac{\kappa_{i_0}-\kappa_{j_0}}{2(\mu_{i_0}-\mu_{j_0})^2},$$
and
$$V=\frac{\mu_{j_0}\kappa_{i_0}-\mu_{i_0}\kappa_{j_0}}{\mu_{j_0}-\mu_{i_0}}+\frac{(\kappa_{i_0}-\kappa_{j_0})^2}{4(\mu_{i_0}-\mu_{j_0})^2}=h_{i_0j_0}(\mu_{i_0j_0}).$$
Similar to the proof of Case (2) in Theorem \ref{th32}, we know in this case $0\leq\lambda_{i_0}\leq 1$.

 (2) $C\leq 0$. We consider $\bar{\bbk}=\bbk-C+1$, then $\bar{\kappa_i}-\mu_i^2>0$, $1\leq i\leq K$. Noting that $\bbl^T\bar{\bbk}=\bbl^T\bbk-C+1$, the optimal $\bbl^*$ is same as in the case of $C>0$.
\end{proof}
\begin{remark}
If $\kappa_i-\mu_i^2>0$, $1\leq i\leq K$, then $V$ can be regarded as the upper variance of some random variable $X$ under the set of probability measures $\{P_1,\cdots, P_K\}$ with $E_{P_i}[X]=\mu_i$ and $E_{P_i}[X^2]=\kappa_i$, $1\leq i\leq K$. Indeed, we can take $X$ being normally distributed under each $P_i$ with $P_i\sim N(\mu_i,\kappa_i-\mu_i^2)$, $1\leq i\leq K$.
\end{remark}


\bibliographystyle{amsplain}

\end{document}